\documentclass[reqno,12pt,letterpaper]{amsart}
\usepackage{amsmath,amssymb,amsthm,graphicx,mathrsfs,url}
\usepackage[usenames,dvipsnames]{color}
\usepackage[colorlinks=true,linkcolor=Red,citecolor=Green]{hyperref}
\usepackage{amsxtra}

\def\smallsection#1{\smallskip\noindent\textbf{#1}.}
\let\epsilon=\varepsilon 

\newtheorem{theo}{Theorem}
\newtheorem{prop}{Proposition}[section]
\newtheorem{defi}[prop]{Definition}

\newtheorem{lemm}[prop]{Lemma}

\numberwithin{equation}{section}

\DeclareMathOperator{\comp}{comp}

\DeclareMathOperator{\End}{End}

\DeclareMathOperator{\Id}{Id}

\let\Im=\Imag
\DeclareMathOperator{\loc}{loc}

\let\Re=\Real
\DeclareMathOperator{\sgn}{sgn}

\DeclareMathOperator{\supp}{supp}

\DeclareMathOperator{\WF}{WF}

\def\WFh{\WF_h}

\title{Spectral gaps for normally hyperbolic trapping}
\author{Semyon Dyatlov}
\email{dyatlov@math.mit.edu}
\address{Department of Mathematics, Massachusetts Institute of Technology,
77 Massachusetts Ave, Cambridge, MA 02139}

\begin{document}

\begin{abstract}
We establish a resonance free strip for codimension 2 symplectic normally hyperbolic trapped sets
with smooth incoming/outgoing tails. An important application is wave decay on Kerr and Kerr--de Sitter black holes.
We recover the optimal size of the strip and give an $o(h^{-2})$ resolvent bound there.
We next show existence of deeper resonance free strips under the $r$-normal hyperbolicity
assumption and a pinching condition. We also give a lower bound on the one-sided cutoff resolvent on the real line.
\end{abstract}

\maketitle

\addtocounter{section}{1}
\addcontentsline{toc}{section}{1. Introduction}

This paper is a collection of results regarding resonance free strips (also
known as \emph{spectral gaps}) and resolvent
estimates in the presence of \emph{normally hyperbolic trapping}. 
Such trapping has received a lot of attention recently because of its connection
with exponential decay of waves on black hole backgrounds,
exponential decay of correlations for contact Anosov flows, and applications to molecular chemistry.
See below for an introduction to the role of resolvent bounds in decay estimates.

In~\cite{wu-zw}, Wunsch and Zworski showed existence of a small spectral gap for symplectic normally
hyperbolic trapped sets under the assumption that the incoming/outgoing tails $\Gamma_\pm$
are smooth and have codimension 1 in the phase space; this includes subextremal Kerr and
Kerr--de Sitter black holes. More recently, Nonnenmacher and Zworski~\cite{no-zw}
extended this result by (a) assuming much weaker regularity of
$\Gamma_\pm$, namely that their tangent spaces at the trapped
set have merely continuous dependence on the base point (b) making no assumption
on the codimension of $\Gamma_\pm$ (c) establishing the optimal size of the gap. The assumptions
of~\cite{no-zw} apply to more general situations, including contact Anosov flows. 

The first result of this paper (Theorem~\ref{t:gap1})
is a spectral gap of optimal size under the assumptions of~\cite{wu-zw}
(and the orientability of $\Gamma_\pm$). The novelty compared to~\cite{no-zw}
is the $o(h^{-2})$ resolvent bound in this gap. The short proof presented here is also more direct
than those in~\cite{wu-zw,no-zw}
because it relies on regular semiclassical analysis rather than
exotic symbol calculus, at a cost of not recovering the extensions (a), (b) discussed above.

We next show (Theorem~\ref{t:gap2}) the existence of deeper resonance free strips, under the additional assumptions of
\emph{$r$-normal hyperbolicity} and \emph{pinching}. The second gap and a Weyl law for resonances in between
the two gaps were previously proved in~\cite{nhp} under the same assumptions. The present paper
essentially removes the projector~$\Pi$ from the method of~\cite{nhp}, 
working directly with the pseudodifferential operators~$\Theta^+_j$ instead (see the discussion
of the proof below).
This shows the existence of additional resonance free strips, but does
not recover the Weyl law or the structure of $\Pi$ (which is important in understanding wave
decay, see~\cite{kdsu}). For a different yet related setting of Pollicott--Ruelle resonances for contact Anosov flows,
existence of multiple gaps under a pinching condition and a Weyl law
for the first band of resonances was proved in~\cite{fa-ts3}.
Our methods also bear some similarities to the recent work~\cite{dfg}
on resonances for the geodesic flow on hyperbolic quotients; the horocyclic
operators $\mathcal U_-$ of~\cite{dfg} play the same role as the operators $\Theta^+_j$
in this paper.

We finally show a lower bound of $h^{-1}\sqrt{\log(1/h)}$ on the one-sided cutoff resolvent
on the real line, complementing the upper bounds of~\cite{bu-zw,da-va3}~-- see~\eqref{e:real-bound-cut}
and~\S\ref{s:lower}.

\smallsection{Motivation}
We first give a brief introduction to resolvent estimates
on a model example.
Assume that $(M,g)$ is a Riemannian manifold isometric to the Euclidean space $\mathbb R^n$
outside of a compact set, and $n$ is odd. Then~\cite[\S4.2]{res} the resolvent
$$
R_g(\omega):=(-\Delta_g-\omega^2)^{-1}:L^2(M)\to H^2(M),\quad \Im\omega>0
$$
admits a meromorphic continuation to $\omega\in\mathbb C$ as a family of operators
$L^2_{\comp}(M)\to H^2_{\loc}(M)$, and its poles are called \emph{resonances}.
We say that $R_g$ has a \emph{spectral gap of size $\nu>0$ with loss of $m\geq 0$ derivatives},
if there exists $C_0>0$ such that
\begin{equation}
  \label{e:lalagap}
\|\chi R_g(\omega)\chi\|_{L^2\to L^2}\leq C_\chi|\omega|^{-1+m},\quad
|\Re\omega|\geq C_0,\quad
\Im\omega\in [-\nu,1]
\end{equation}
for all $\chi\in C_0^\infty(M)$. If~\eqref{e:lalagap} holds, then each solution to the wave equation
$$
(\partial_t^2-\Delta_g)u=0,\quad
u|_{t=0}=\chi u_0\in H^1_{\comp}(M),\quad
u_t|_{t=0}=\chi u_1\in L^2_{\comp}(M)
$$
satisfies the \emph{resonance expansion} (assuming there are no resonances on $\{\Im\omega=-\nu\}$)
$$
u(t,x)=\sum_j \sum_{k=0}^{J(\omega_j)-1} t^k e^{-it\omega_j} u_{jk}(x)+ u_R(t,x),
$$
with the sum above over the finitely many resonances in $\{\Im\omega\geq -\nu\}$ with multiplicities
$J(\omega_j)$, and the remainder is exponentially decaying in $t\geq 0$:
$$
\|e^{\nu t} \chi u_R\|_{H^1_{t,x}(\{t>0\})}\leq C_\chi(\|u_0\|_{H^{1+m}}+\|u_1\|_{H^m}).
$$
See for instance~\cite[Proposition~2.1]{xpd} for the proof. We see that the size $\nu$ of the spectral
gap gives the rate of exponential decay of the remainder, while $m$ is the number of derivatives lost
in the estimate. Whether or not~\eqref{e:lalagap} holds depends on the structure of the \emph{trapped set}
for the geodesic flow $\varphi_t:T^*M\to T^*M$
$$
K=\{(x,\xi)\in T^* M\mid |\xi|_g\in (1/2,2),\ \exists V\Subset T^*M:
\varphi_t(x,\xi)\in V\quad\text{for all }t\in\mathbb R\}
$$
The bound~\eqref{e:lalagap} is equivalent to the estimate
\begin{equation}
  \label{e:lalagaph}
\|\chi R_h(\lambda)\chi\|_{L^2\to L^2}\leq C_\chi h^{-1-m},\quad
\Re\lambda=\pm 1,\quad
\Im\lambda\in [-\nu h,h],\quad
0<h\ll 1
\end{equation}
where $R_h(\lambda)=h^{-2}R_g(h^{-1}\lambda)$ is the semiclassical resolvent:
\begin{equation}
  \label{e:scatres}
R_h(\lambda)=(-h^2\Delta_g-\lambda^2)^{-1}:L^2(M)\to H^2(M),\quad
\Im\lambda>0.
\end{equation}
The bounds~\eqref{e:lalagaph} for $\Re\lambda=1$ and $\Re\lambda=-1$ are equivalent
(by taking adjoints), thus we assume that $\Re\lambda=1$.

By the gluing method of
Datchev--Vasy~\cite[\S4.1]{da-va1}, it suffices to prove the bound for the model resolvent
\begin{equation}
  \label{e:modelb}
\|(-h^2\Delta_{g'}-\lambda^2-iQ')^{-1}\|_{L^2(X)\to L^2(X)}\leq C h^{-1-m},
\end{equation}
where $(X,g')$ is a compact Riemannian manifold which contains a part $X_M$ isometric to $(M\cap B_{R_0},g)$
where $B_{R_0}$ is the ball of radius $R_0\gg 1$ in $\mathbb R^n$, and
$Q'\in\Psi^2_h(X)$ is a semiclassical pseudodifferential operator
(see~\S\ref{s:prelim} for notation) such that $\sigma(Q')\geq 0$ everywhere
and $Q'$ is elliptic on $X\setminus X_M$ and supported away from $X_M\cap B_{R_0/2}$.
The operator $Q'$, as well as the operators $Q'',Q$ introduced below, are called \emph{complex absorbing operators}
and generalize complex absorbing potentials used for instance in quantum chemistry.

To make the setup of~\eqref{e:R} apply, we need
to make the dependence on $\lambda$ linear. For that, take self-adjoint compactly microlocalized operators
$P,Q''\in\Psi^{\comp}_h(X)$ such that
$$
-h^2\Delta_{g'}-iQ'=(P-iQ'')^2+\mathcal O(h^\infty)\quad\text{microlocally on }
\{1/2<|\xi|_{g'}<2\}.
$$
See for instance~\cite[Lemma~4.6]{g-s}; moreover, we have in $\{1/2<|\xi|_{g'}<2\}$,
$$
\sigma(P)-i\sigma(Q'')=\sqrt{|\xi|_{g'}^2-i\sigma(Q')},
$$
where $\sqrt{\cdot}$ maps positive numbers to positive numbers,
so in particular $\sigma(Q'')\geq 0$ everywhere.
Then (see~\cite[Lemma~4.3]{nhp} for a more general argument)
$$
-h^2\Delta_{g'}-\lambda^2-iQ'=(P-iQ''+\lambda)(P-iQ''-\lambda)\quad\text{microlocally on }
\{1/2<|\xi|_{g'}<2\}.
$$
By the elliptic estimate~\cite[Proposition~2.4]{zeta}, it suffices to prove a bound
on $A(-h^2\Delta_{g'}-\lambda^2-iQ')^{-1}$ for some $A\in\Psi^{\comp}_h(X)$ which
is elliptic on
$$
\Omega:=\{|\xi|_{g'}^2-i\sigma(Q')=1\}\subset \{x\in X_M,\ |\xi|_g=1\}.
$$
Since $P-iQ''+\lambda$ is elliptic on $\Omega$,
we reduce~\eqref{e:modelb} to the bound
\begin{equation}
  \label{e:modelb2}
\|(P-iQ-\lambda)^{-1}\|_{L^2(X)\to L^2(X)}\leq Ch^{-1-m},
\end{equation}
where $Q\in\Psi^0_h(X)$, $\sigma(Q)\geq 0$ everywhere, and
$Q=Q''$ microlocally near $\Omega$. In this paper,
we prove the bound~\eqref{e:modelb2}
for the case when the trapped set $K$ has a \emph{normally hyperbolic}
structure.

\pagebreak

\smallsection{Setup}
Let $X$ be a compact manifold with a fixed volume form and
\begin{equation}
  \label{e:R}
R(\lambda;h):= (P(h)-iQ(h)-\lambda)^{-1}:L^2(X)\to L^2(X),
\end{equation}
where $P(h),Q(h)\in\Psi^{0}_h(X)$ (henceforth denoted simply $P,Q$) are self-adjoint
semiclassical pseudodifferential operators of order $0$ with principal
symbols $p,q$ (see~\S\ref{s:prelim} for
the semiclassical notation used) and $\lambda\in\mathbb C$ satisfies
$\lambda=\mathcal O(h)$. We furthermore
assume that $q\geq 0$ everywhere and $q(x,\xi)>0$ for $|\xi|$ large enough.
The family $R(\lambda;h)$ is meromorphic in $\lambda$,
and we call its poles \emph{resonances}.
(The resonances for the model resolvent~\eqref{e:R} need not coincide with
resonances for the scattering resolvent~\eqref{e:scatres}, but a spectral gap
for~\eqref{e:R} implies a spectral gap for~\eqref{e:scatres}.)

\begin{figure}
\includegraphics{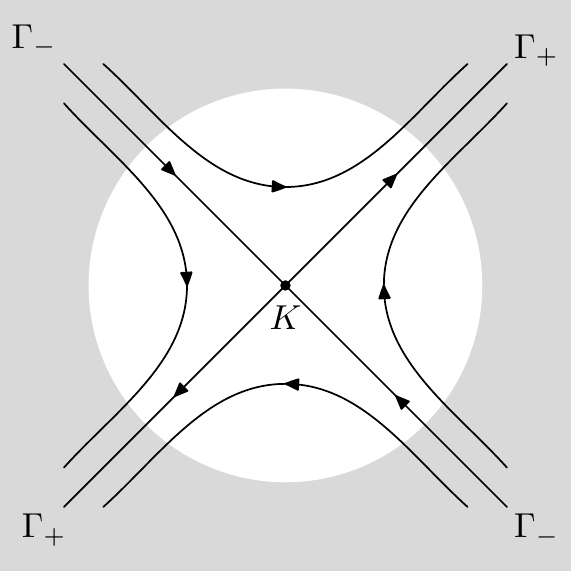}
\qquad
\includegraphics{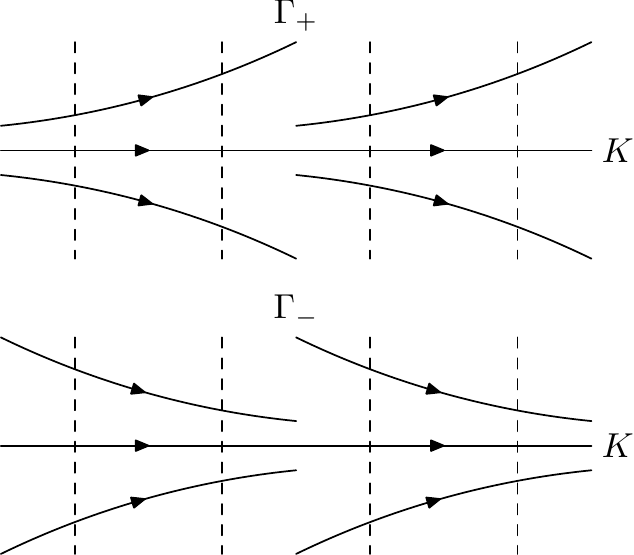}
\caption{An illustration of normally hyperbolic trapping, with the Hamiltonian flow
lines of $p$ on the whole $T^*X$ shown on the left and the restriction of the flow
to $\Gamma_\pm$, on the right. The shaded region is the set $\{q>0\}$ where
the complex absorbing operator takes over. The dashed lines are the flow
lines of $H_{\varphi_\pm}$, see \S\ref{s:dynamics}.}
\label{f:trapping}
\end{figure}

As a particular application,
\cite{vasy1} shows exponential decay of linear waves on Kerr--de Sitter black holes
(modulo a finite dimensional space) using a resolvent estimate of type~\eqref{e:main}.
We refer the reader to~\cite{vasy1,kdsu} for further discussion
of the relation between resolvent bounds and wave decay, and for an overview
of previous results on wave decay for black holes.
We remark that the analysis of trapping is critical for the understanding of semilinear and quasilinear
wave equations on black holes~-- see the work of Hintz--Vasy~\cite{hi-va1,hi-va2}. In particular, the loss
of regularity in the decay estimate (quantified by $m>0$ in~\eqref{e:lalagap}) is
the reason why~\cite{hi-va2} has to invoke Nash--Moser theory.

We also point out that normally hyperbolic (in fact, $r$-normally hyperbolic) trapped sets
appear naturally in the semiclassical theory of chemical reaction dynamics, see \cite{gsww}
for a physical description and~\cite[Remark~1.1]{no-zw} for a mathematical explanation.

Our results also hold (with the same proofs) in the general
framework of~\cite[\S4.1]{nhp}, which does not use
complex absorbing operators and applies to a variety of scattering problems.
In fact, \cite[\S\S4.2, 8.1]{nhp} reduces the general case microlocally to a neighborhood of the trapped set,
providing an alternative to the gluing method discussed in the motivation section above.

We make the following normally hyperbolic trapping assumptions (see Figure~\ref{f:trapping};
these assumptions provide a definition of $\Gamma_\pm$ and $K$ in our framework):
\begin{enumerate}
\item \label{a:1}
$\Gamma_\pm$ are codimension 1 orientable $C^\infty$ submanifolds of $T^*X$
such that $\Gamma_\pm\cap \{p=0\}\cap\{q=0\}$ are compact;
\item  if $(x,\xi)\in\{p=0\}\setminus\Gamma_\pm$, then $e^{\mp tH_p}(x,\xi)\in \{q>0\}$
for some $t\geq 0$;
\item the Hamiltonian field $H_p$ is tangent to $\Gamma_\pm$;
\item $\Gamma_\pm$ intersect transversely, $K=\Gamma_+\cap \Gamma_-$ is called the \emph{trapped set}
and we assume that $\WFh(Q)\cap K\cap \{|p|\leq\delta\}=\emptyset$
for $\delta>0$ small enough;
\item $K$ is a symplectic codimension 2 submanifold of $T^*X$;
\item \label{a:5}
if $v\in T_K\Gamma_\pm$, then $de^{\mp tH_p}\cdot v$ exponentially approaches
$TK\subset T_K\Gamma_\pm$ as $t\to +\infty$.
\end{enumerate}
Assumptions~\eqref{a:1}--\eqref{a:5} hold for subextremal
Kerr and Kerr--de Sitter black holes with small cosmological constant,
see~\cite[\S2]{wu-zw}, \cite[\S6.4]{vasy1},
and~\cite[\S3.2]{kdsu}. Note that, as in~\cite{nhp}, $\Gamma_\pm$
are open subsets of the full incoming/outgoing tails of the flow cut off
to a small neighborhood of $\{p=0\}\cap \{q=0\}$ (and thus are noncompact manifolds without boundary);
as in~\cite{vasy1} (see also~\cite[\S3.5]{kdsu}), one needs to embed the Kerr(--de Sitter) trapped set into a compact manifold
without boundary.

We furthermore define $0<\nu_{\min}\leq\nu_{\max}$ as the maximal and the minimal
numbers such that for each $\varepsilon>0$ there exists a constant $C$ such that
for each $v\in T_K\Gamma_\pm$
$$
C^{-1}e^{-(\nu_{\max}+\varepsilon)t} |\pi(v)|
\leq |\pi(de^{\mp tH_p}\cdot v)|\leq
Ce^{-(\nu_{\min}-\varepsilon)t} |\pi(v)|,\quad
t\geq 0,
$$
where $\pi:T_K\Gamma_\pm\to T_K\Gamma_\pm$ is any fixed smooth linear projection map
whose kernel is equal to $TK$.
In other words, $\nu_{\min}$ and $\nu_{\max}$ are the minimal and
maximal expansion rates in directions transversal to the trapped set.


\smallsection{Results}
Our first result is a resonance free strip with a polynomial resolvent bound:
\begin{theo}
  \label{t:gap1}
For each $\varepsilon>0$ and $h$ small enough depending on $\varepsilon$
(see Figure~\ref{f:gaps})
\begin{equation}
  \label{e:main}
\|R(\lambda)\|_{L^2(X)\to L^2(X)}=o(h^{-2})\quad\text{if }
|\lambda|=\mathcal O(h),\
\Im\lambda>-(\nu_{\min}-\varepsilon)h/2.
\end{equation}
\end{theo}
\noindent\textbf{Remarks}.
(i) Theorem~\ref{t:gap1} also extends to the case when $P,Q$ are operators
acting on sections of some vector bundle $\mathcal E$ over $X$, as
long as they are self-adjoint with respect to some smooth inner product on
the fibers of $\mathcal E$, and their principal symbols (see for instance~\cite[\S C.1]{zeta}
for a definition) are equal to $p \cdot \Id_{\mathcal E},q \cdot \Id_{\mathcal E}$
for $p,q\in C^\infty(X)$. 
We can also relax the assumption that $Q=\mathcal O(h^\infty)$ microlocally near $K\cap \{p=0\}$,
requiring instead that $q=0$ 
and $\sigma(h^{-1}Q)\geq 0$ near $K\cap \{p=0\}$ (as an endomorphism of $\mathcal E$). Indeed, the proof of Lemma~\ref{l:theta-killed}
still applies, and the relation~\eqref{e:measure-below} becomes a lower
bound on $\mu(e^{-tH_p}(U_\delta))$. We leave the details to the reader.

(ii) We note that polynomial resolvent bounds are known in a variety of other situations~--
see the review of Wunsch~\cite{wu} for an overview of the subject and~\cite{chr} for other recent results.
Also the much earlier work of G\'erard--Sj\"ostrand~\cite{ge-sj} treated normally hyperbolic trapping
in the analytic category with (implicit) exponential resolvent bounds.

\medskip

\begin{figure}
\includegraphics{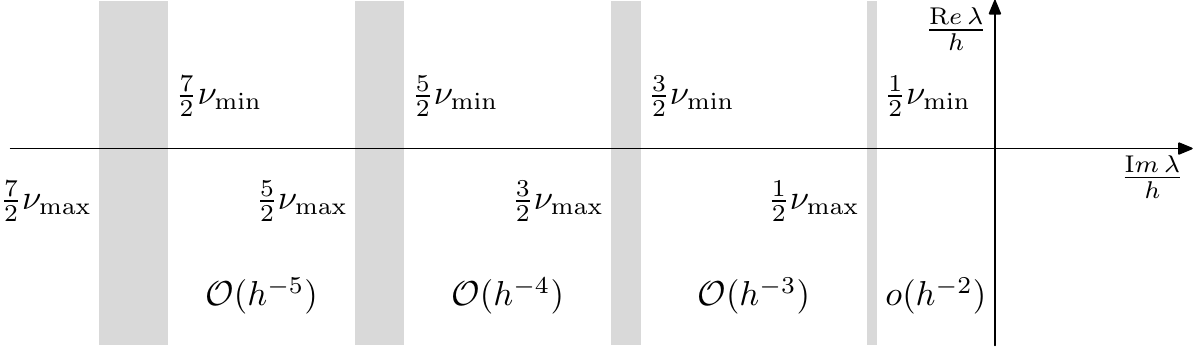}
\caption{The resonance free strips of Theorems~\ref{t:gap1} and~\ref{t:gap2}
(ignoring $\varepsilon$) are in white, with bounds on the resolvent $R(\lambda)$ written in the strips
and (potential) bands of resonances in between
the strips shaded.}
\label{f:gaps}
\end{figure}

We now make the stronger assumption that the trapping is \emph{$r$-normally hyperbolic} for large $r$,
namely
\begin{equation}
  \label{e:r-nh}
\nu_{\min}>r\mu_{\max},
\end{equation}
where $\mu_{\max}$ is the maximal expansion rate along $K$, namely
$\|de^{tH_p}|_{TK}\|=\mathcal O(e^{(\mu_{\max}+\varepsilon)t})$ as
$|t|\to\infty$, for all $\varepsilon>0$. The condition~\eqref{e:r-nh}
means that the rate of expansion in directions transversal to $K$ is much larger
than the rate of expansion along $K$, and it holds for Kerr(--de Sitter) black holes
(see the references above).

Under the assumption~\eqref{e:r-nh}
and an appropriate pinching condition, we exhibit deeper resonance free
strips with polynomial resolvent bounds:
\begin{theo}
  \label{t:gap2}
Assume that the trapping is $r$-normally hyperbolic for all $r$.
Fix $m\geq 1$ and assume the pinching condition
$$
(m+1/2)\nu_{\min}>(m-1/2)\nu_{\max}.
$$
Then for each small $\varepsilon>0$ and $h$ small enough depending on $\varepsilon$, we have (see Figure~\ref{f:gaps})
\begin{equation}
  \label{e:gap2}
\begin{gathered}
\|R(\lambda)\|_{L^2(X)\to L^2(X)}\leq Ch^{-m-2}\quad\text{if }\\
|\lambda|=\mathcal O(h),\
h^{-1}\Im\lambda\in [-(m+1/2)\nu_{\min}+\varepsilon,-(m-1/2)\nu_{\max}-\varepsilon].
\end{gathered}
\end{equation}
\end{theo}

\noindent\textbf{Remarks}. (i) Since semiclassical arguments
only require finitely many derivatives, it is straightforward to see that Theorem~\ref{t:gap2} is also
valid for trapping which is $r$-normally hyperbolic for some fixed $r$, where $r$ is chosen large
enough depending only on $m$ and on the dimension of $X$. Such trapping is structurally stable
under perturbations of the symbol $p$, see~\cite{HPS} and~\cite[\S5.2]{nhp}.

(ii) Unlike Theorem~\ref{t:gap1}, we only prove Theorem~\ref{t:gap2} in the scalar case,
rather than for operators on general vector bundles, and for $Q=\mathcal O(h^\infty)$
near $K\cap \{p=0\}$.

(iii) The pinching condition is true (for all $m$) for the Schwarzschild(--de Sitter)
black hole. In the Kerr case, it was studied 
for $m=1$ in~\cite[Figure 2(a) and \S3.3]{kdsu}.

(iv) A more careful argument, using in part the proof of
Theorem~\ref{t:gap1}, shows that the bound~\eqref{e:gap2} could
be improved to $\mathcal O(h^{-m-1})$ for $h^{-1}\Im\lambda\geq -m\nu_{\min}+\varepsilon$
and to $o(h^{-m-2})$ elsewhere in the strip, and more optimal bounds can be obtained
using complex interpolation. The transition from one bound to the other at $h^{-1}\Im\lambda\geq -m\nu_{\min}$ is similar in spirit to the transition from the
nontrapping bound $\mathcal O(h^{-1})$ in the upper half-plane to the $o(h^{-2})$ bound
in the first gap. The method of Theorem~\ref{t:gap1} produces the slightly better bound of
$o(h^{-m-2})$ because Lemma~\ref{l:lipschiz} applies under slightly more relaxed
conditions on the norm of $\Theta^+u$ than Lemma~\ref{l:lab}.

\smallsection{Ideas of the proofs}
The proof of Theorem~\ref{t:gap2} is based on constructing
pseudodifferential operators $\Theta^+_j,W^+_j$ such that microlocally near $K\cap \{p=0\}$,
we have
$$
\Theta^+_m\dots\Theta^+_0 (P-\lambda) = (P-ihW^+_{m+1}-\lambda)\Theta^+_m\dots\Theta^+_0+\mathcal O(h^\infty).
$$
The principal symbol of $W^+_{m+1}$ is bounded from below
by $(m+1)(\nu_{\min}-\varepsilon)$ near $K\cap \{p=0\}$, which has the effect of shifting the spectral
parameter $\lambda$ into the upper half-plane. Since there are no resonances in the upper half-plane, we obtain for $m$ large enough depending on $\Im\lambda$,
if $(P-iQ-\lambda)u=\mathcal O(h^\infty)$, then
\begin{equation}
  \label{e:new-eq}
\Theta^+_m\dots \Theta^+_0 u=\mathcal O(h^\infty)\quad\text{microlocally near }K\cap \{p=0\}.
\end{equation}
This is a pseudodifferential equation on $u$. The principal symbol of each $\Theta^+_j$ is
a defining function $\varphi_+$ of $\Gamma_+$, and the equation~\eqref{e:new-eq}
gives information about the phase space behavior of $u$ along the Hamiltonian
flow lines of $\varphi_+$, which are transversal to
$K$ (see Figure~\ref{f:trapping}). Together with the information about the decay of $u$ as one approaches $K$,
which depends on $\Im\lambda$, one can show that $\lambda$ has to avoid the strips in~\eqref{e:gap2}.

A good example to keep in mind, considered in \S\ref{s:lower}, is
$$
P_0={h\over i}(x\partial_x+1/2).
$$
The relevant solutions to the equation $(P_0-\lambda)u=0$ are those
whose wavefront set lies on $\Gamma_+=\{\xi=0\}$, which in particular
means that $u$ has to be smooth. This gives the resonances
$\lambda_m=-(m+{1\over 2})ih$, $m=0,1,\dots$ with
resonant states $u_m(x)=x^m$ and the operator
$\partial_x$ maps $u_m$ to
$mu_{m-1}$. For the case of $P_0$, we can
then take $\Theta^+_j={h\over i}\partial_x$ and $W^+_j=j$.

In the case of Theorem~\ref{t:gap1}, we cannot construct the operators $\Theta^+_j$, however
a rougher version of the operator $\Theta^+_0$ gives enough information on the concentration
of $u$ along the flow lines of $H_{\varphi_+}$ to still obtain an estimate.
The proof of Theorem~\ref{t:gap1} in given in \S\ref{s:gap1} and
the proof of Theorem~\ref{t:gap2} is given in \S\ref{s:gap2}.
Both sections use preliminary dynamical and analytical constructions
of \S\ref{s:prelim}.

\smallsection{Bounds on the real line}
For $\lambda$ on the real line, it is shown in~\cite{wu-zw} that
\begin{equation}
  \label{e:real-bound}
\|R(\lambda)\|_{L^2\to L^2}\leq Ch^{-1}\log(1/h),\quad\text{if }|\lambda|=\mathcal O(h),\ \Im\lambda=0.
\end{equation}
The results of Burq--Zworski~\cite{bu-zw} and Datchev--Vasy~\cite{da-va3}
(see also the papers of Burq~\cite{Burq1,Burq2}, Vodev~\cite{Vodev}, and Datchev~\cite{Datchev} for bounds
on the two-sided cutoff resolvent) deduce the following
improved cutoff resolvent bound:
\begin{equation}
  \label{e:real-bound-cut}
\|R(\lambda)A\|_{L^2\to L^2}\leq Ch^{-1}\sqrt{\log(1/h)},\quad\text{if }|\lambda|=\mathcal O(h),\ 
\Im\lambda=0,
\end{equation}
if $A$ is an $h$-pseudodifferential operator such that $\WFh(A)\cap K=\emptyset$.
The bound~\eqref{e:real-bound-cut} can be improved to $\mathcal O(h^{-1})$
(which is the estimate in the case there is no trapping) if one puts $A$ on both sides
of $R(\lambda)$, or if the principal symbol of $A$
vanishes on $\Gamma_-$~-- see~\cite{da-va2,hi-va3}.

In~\S\ref{s:lower}, we show that the bound~\eqref{e:real-bound-cut} is sharp
by giving an example of an operator $P-iQ$ and $A$ for which the corresponding
lower cutoff resolvent bound holds.

\section{Preliminaries}
\label{s:prelim}

The proofs in this paper rely on the methods of \emph{semiclassical analysis}.
We refer the reader to~\cite{e-z} for an introduction to the subject, and
to~\cite[Appendix~C.2]{zeta} for the notation used here. In particular,
we use the algebra $\Psi_h^0(X)$ of semiclassical pseudodifferential operators
with symbols in the class $S^0(X)$. We moreover require that symbols
of the elements of $\Psi_h^0(X)$ are \emph{classical} in the
sense that they have an asymptotic expansion in nonnegative
integer powers of $h$. Denote by
$$
\sigma:\Psi_h^0(X)\to S^0(X)
$$
the principal symbol map
and recall the standard identities for $A,B\in\Psi_h^0(X)$,
\begin{equation}
  \label{e:identities}
\sigma(AB)=\sigma(A)\sigma(B),\quad
\sigma(A^*)=\overline{\sigma(A)},\quad
\sigma(h^{-1}[A,B])=-i\{\sigma(A),\sigma(B)\}.
\end{equation}
Moreover, each $A\in\Psi_h^0(X)$ acts $L^2(X)\to L^2(X)$ with norm bounded uniformly in $h$.
In fact, we have the bound~\cite[Theorem~5.1]{e-z}
\begin{equation}
  \label{e:norm-psi}
\limsup_{h\to 0}\|A\|_{L^2(X)\to L^2(X)}\leq \|\sigma(A)\|_\infty:=\sup_{T^*X}|\sigma(A)|.
\end{equation}

We will mostly use the subalgebra of \emph{compactly microlocalized} pseudodifferential
operators, $\Psi_h^{\comp}(X)\subset \Psi_h^0(X)$. For $A\in\Psi_h^{\comp}(X)$, its
\emph{wavefront set} $\WFh(A)$ is a compact subset of $T^*X$. We say that
$A=B+\mathcal O(h^\infty)$ microlocally in some set $U\subset T^*X$ if $\WFh(A-B)\cap U=\emptyset$.

We will often consider sequences $u(h_j)\in L^2(X)$,
where $h_j\to 0$ is a sequence of positive numbers; we typically suppress the
dependence on $j$ and simply write $u=u(h_j)$ and $h=h_j$.
For such a sequence, we say that $u=O(h^\delta)$ \emph{microlocally} in some open
set $U\subset T^*X$, if $\|Au\|_{L^2}=O(h^\delta)$ for each $A\in\Psi^{\comp}_h(X)$
such that $\WFh(A)\subset U$. The notion of $u=o(h^\delta)$ in $U$ is defined similarly.

We may also consider distributions in $L^2(X;\mathcal E)$, where $\mathcal E$ is some
smooth vector bundle over $X$ with a prescribed inner product. One can
take pseudodifferential operators acting on sections of $\mathcal E$,
see~\cite[Appendix~C.1]{zeta}, and the principal symbol of such an operator
is a section of the endomorphism bundle $\End(\mathcal E)$ over $T^*X$. However,
we will only consider \emph{principally scalar} operators, that is,
operators whose principal symbols have the form $a(x,\xi)\Id_{\mathcal E}$.
(Equivalently, these are the operators $A\in\Psi^k(X;\End(\mathcal E))$ such that
 for all $B\in\Psi^k(X;\End(\mathcal E))$, we have
$[A,B]=\mathcal O(h)$.)
The formulas~\eqref{e:identities} still hold for principally scalar operators. We henceforth
suppress $\mathcal E$ in the notation.

\subsection{Semiclassical defect measures}

Our proofs rely on the following definition, designed to capture the concentration
of a sequence of $L^2$ functions in phase space:
\begin{defi}
  \label{d:measures}
Assume that $h_j\to 0$ and
$u=u(h_j)\in L^2(X)$ is bounded uniformly in $L^2$ as $j\to\infty$.
We say that $u$ converges to a nonnegative Radon measure $\mu$ on $T^*X$ if for each
$A=A(h)\in \Psi^{\comp}_h(X)$, we have
\begin{equation}
  \label{e:measures}
\langle A(h_j) u(h_j),u(h_j)\rangle\to\int_{T^*X}\sigma(A)\,d\mu\quad\text{as }j\to\infty.
\end{equation}
\end{defi}
See~\cite[Chapter~5]{e-z} for an introduction to defect measures; in particular,
\begin{itemize}
\item for each sequence $u(h_j)$ uniformly bounded in $L^2$, there exists a subsequence $h_{j_k}$ such that
$u(h_{j_k})$ converges to some $\mu$~\cite[Theorem~5.2]{e-z};
\item we have $\mu(U)=0$ for some open $U\subset T^*X$ if and only if $u=o(1)$ microlocally on $U$.
\end{itemize}
We now show the basic properties of semiclassical measures corresponding to approximate
solutions of differential equations:
\begin{lemm}
  \label{l:measure-elliptic} (Ellipticity)
Take $P\in\Psi^0_h(X)$ and denote $p=\sigma(P)$.
Assume that $u=u(h_j)$ converges to some measure $\mu$ and
$Pu=o(1)$ microlocally in some open set $U\subset T^*X$. Then
$\mu(U\cap \{p\neq 0\})=0$.
\end{lemm}
\begin{proof}
By~\eqref{e:identities}, we have for each $A\in \Psi_h^{\comp}(X)$ with $\WFh(A)\subset U$,
$$
\int_{T^*X} \sigma(A)p\,d\mu=\lim_{h\to 0} \langle APu,u\rangle=0.
$$
Since $\sigma(A)p$ can be any function in $C_0^\infty(U\cap \{p\neq 0\})$, we have
$\mu(U\cap \{p\neq 0\})=0$.
\end{proof}

\begin{lemm}
  \label{l:typical} (Propagation)
Take $P,W\in\Psi^0_h(X)$, denote $p=\sigma(P)$, $w=\sigma(W)$, and assume that
$P^*=P$. Assume that $u=u(h_j)$ converges to some measure $\mu$
and denote $f:=(P-ihW)u$. Then for each $a\in C_0^\infty(T^*X)$
and for each $Y\in\Psi_h^{\comp}(X)$ such that
$Y=1+\mathcal O(h^\infty)$ microlocally in a neighborhood of $\supp a$,
\begin{equation}
  \label{e:typical}
\bigg|\int_{T^*X} (H_p-2\Re w)a\,d\mu\bigg|\leq 2\|a\|_{\infty}\cdot \limsup_{h\to 0} (h^{-1}\|Yf\|_{L^2}\cdot \|Yu\|_{L^2}).
\end{equation}
In particular, if $f=o(h)$ microlocally in a neighborhood of $\supp a$, then
$$
\int_{T^*X} (H_p-2\Re w)a\,d\mu=0.
$$
\end{lemm}
\begin{proof}
Take $A\in\Psi^{\comp}_h(X)$ with
$\sigma(A)=a$ and $Y=1+\mathcal O(h^\infty)$
near $\WFh(A)$. Without loss of generality, we may assume that
$W^*=W$; indeed, one can put the imaginary part of $W$ into $P$.
Then we have
\begin{equation}
\label{e:bazinga}
\begin{gathered}
{\langle A f,u\rangle-\langle Au,f\rangle\over ih}=
{\langle A (P-ihW)u,u\rangle-\langle (P+ihW)Au,u\rangle\over ih}\\
=\big\langle \big((ih)^{-1}[A,P]-(AW+WA)\big)u,u\big\rangle.
\end{gathered}
\end{equation}
By~\eqref{e:identities}, $\sigma\big((ih)^{-1}[A,P]-(AW+WA)\big)=H_pa-2wa$; therefore,
the right-hand side of~\eqref{e:bazinga} converges in absolute value to
the left-hand side of~\eqref{e:typical}. The left-hand side of~\eqref{e:bazinga}
is equal to
$$
{\langle AYf,Yu\rangle-\langle AYu,Yf\rangle\over ih}+\mathcal O(h^\infty),
$$
and its limit as $h\to 0$ is bounded by the right-hand side of~\eqref{e:typical}
by~\eqref{e:norm-psi}.
\end{proof}

\subsection{Dynamical preliminaries}
\label{s:dynamics}

In this section, we review several properties of normally hyperbolic trapping,
assuming that $p,q,\Gamma_\pm,K$ satisfy properties~\eqref{a:1}--\eqref{a:5} listed
in the introduction. We start with the following restatement
of~\cite[Lemma~5.1]{nhp}, see also~\cite[Lemma~4.1]{wu-zw}.
To be able to quantize the functions $\varphi_\pm,c_\pm$,
we multiply them by a cutoff to obtain compactly supported functions,
but only require their properties to hold in a neighborhood $U$
of $K\cap \{p=0\}$.
\begin{lemm}
  \label{l:phi+}
Fix small $\varepsilon>0$. Then there exists a bounded neighborhood $U$ of \break $K\cap \{p=0\}$
and functions $\varphi_\pm\in C_0^\infty(T^*X)$ such that $\WFh(Q)\cap U=\emptyset$ and
\begin{enumerate}
\item for $\delta>0$ small enough, the set
\begin{equation}
  \label{e:U-delta}
U_\delta:=\{|\varphi_+|<\delta,\ |\varphi_-|<\delta,\ |p|<\delta\}\cap U
\end{equation}
is compactly contained in $U$;
\item $\Gamma_\pm\cap U=\{\varphi_\pm=0\}\cap U$;
\item $H_p\varphi_\pm=\mp c_\pm\varphi_\pm$ on $U$, where $c_\pm\in C_0^\infty(T^*X)$
satisfy
$$
0<\nu_{\min}-\varepsilon\leq c_\pm\leq \nu_{\max}+\varepsilon\quad\text{on }U;
$$
\item $\{\varphi_+,\varphi_-\}>0$ on $U$.
\end{enumerate}
\end{lemm}
Note that in particular, for $\delta>0$ small enough and all $t\geq 0$,
\begin{equation}
  \label{e:U-delta-map}
e^{-tH_p}(U_\delta\cap \Gamma_+)
\subset \{|\varphi_-|<e^{-(\nu_{\min}-\varepsilon)t}\delta\}\cap U_\delta\cap \Gamma_+.
\end{equation}
Moreover, the map
\begin{equation}
  \label{e:funny-coordinates}
(\rho,s)\mapsto e^{sH_{\varphi_+}}(\rho),\quad
\rho\in K,\ s\in \mathbb R
\end{equation}
is a diffeomorphism from some neighborhood of $(K\cap \{p=0\})\times\{s=0\}$
in $K\times\mathbb R$ onto some neighborhood of $K\cap \{p=0\}$
in $\Gamma_+$.

For Theorem~\ref{t:gap2},
we also need existence of solutions to the transport equation on $\Gamma_+$:
\begin{lemm}
  \label{l:transport}
Assume that the flow is $r$-normally hyperbolic in the sense of~\eqref{e:r-nh}, for
all $r$. Fix small $\delta>0$ and take $U_\delta$ defined by~\eqref{e:U-delta}.
Then for each $f\in C^\infty(\Gamma_+\cap U_\delta)$, there exists
unique $u\in C^\infty(\Gamma_+\cap U_\delta)$ such that
\begin{equation}
  \label{e:transport}
(H_p+c_+) u=f.
\end{equation}
\end{lemm}
\begin{proof}
We use~\cite[Lemma~5.2]{nhp}, which gives existence and uniqueness of
$v\in C^\infty(\Gamma_+\cap U_\delta)$
such that $H_p v=g$ and $v|_K=0$, for each $g\in C^\infty(\Gamma_+\cap U_\delta)$
such that $g|_K=0$.

We claim that there exists a function $G\in C^\infty(\Gamma_+\cap U_\delta)$
such that
$$
H_p F = c_+F,\quad
F:=e^G \{\varphi_+,\varphi_-\}^{-1}\varphi_-.
$$
Indeed, $G$ needs to solve the equation
$$
H_pG = c_+-c_-+{H_p\{\varphi_+,\varphi_-\}\over \{\varphi_+,\varphi_-\}},
$$
and we use~\cite[Lemma~5.2]{nhp} since the right-hand side of this equation vanishes on $K$.

Now~\eqref{e:transport} becomes
$$
H_p (Fu)=Ff,
$$
and it remains to invoke~\cite[Lemma~5.2]{nhp} once again.
\end{proof}

\subsection{Basic estimates}

We now derive some estimates for the operator $P-iQ-\lambda$, where
$P,Q$ satisfy the assumptions of Theorem~\ref{t:gap1}.
We start by using the complex absorbing operator $Q$ to reduce the analysis
to a neighborhood of $K\cap \{p=0\}$, where $K$ is the trapped set: 
\begin{lemm}
  \label{l:concentrate}
Assume that $h_j\to 0$, $u=u(h_j)\in L^2(X)$ is bounded uniformly in $h$, and
$$
\|(P-iQ-\lambda)u\|_{L^2}=\mathcal O(h^{\alpha+1})
$$
for some constant $\alpha> 0$ and $\lambda=\mathcal O(h)$. Then:

1. $u=\mathcal O(h^\alpha)$ microlocally on the complement of $\{p=0\}\cap \Gamma_+$.

2. If $u$ converges to some measure $\mu$ in the sense of Definition~\ref{d:measures},
then $\mu$ is supported on $\{p=0\}\cap \Gamma_+$.
If moreover $\|u\|_{L^2}\geq c>0$ for some constant $c$,
then for each neighborhood $U$
of $\{p=0\}\cap K$, we have $\mu(U)>0$.

Same statements are true if we replace $\mathcal O(h^{\alpha+1}),\mathcal O(h^\alpha)$ by $o(h^{\alpha+1}),o(h^\alpha)$
and $\alpha\geq 0$.
\end{lemm}
\begin{proof}
Take $A\in\Psi^0_h(X)$ such that $\WFh(A)\cap \{p=0\}\cap \Gamma_+=\emptyset$.
Then
\begin{equation}
  \label{e:propest}
\|Au\|_{L^2}\leq Ch^{-1}\|(P-iQ-\lambda)u\|_{L^2}+\mathcal O(h^\infty)\|u\|_{L^2}.
\end{equation}
This follows from the elliptic estimate~\cite[Proposition~2.4]{zeta} and
propagation of singularities with a complex absorbing operator~\cite[Proposition~2.5]{zeta},
since for each $(x,\xi)\in\WFh(A)$, there exists $t\geq 0$
such that $e^{-tH_p}(x,\xi)\in \{p-iq\neq 0\}$.

The estimate~\eqref{e:propest} immediately implies part~1 of the lemma,
as well as the first statement of part~2. To see the second statement of part~2, it
suffices to use the following estimate, valid for each
$B\in\Psi^{\comp}_h(X)$ such that $\sigma(B)\neq 0$ on $\{p=0\}\cap K$:
\begin{equation}
  \label{e:propest2}
\|u\|_{L^2}\leq Ch^{-1}\|(P-iQ-\lambda)u\|_{L^2}+\|Bu\|_{L^2}.
\end{equation}
To prove~\eqref{e:propest2}, we again use the elliptic estimate and propagation
of singularities, noting that for each $(x,\xi)\in T^*X$, there exists $t\geq 0$
that $e^{-tH_p}(x,\xi)\in \{p-iq\neq 0\}\cup \{\sigma(B)\neq 0\}$,
see~\cite[Lemma~4.1]{nhp}.
\end{proof}
The next lemma is a generalization of the statement that there are no resonances
in the upper half-plane (which, keeping in mind Lemma~\ref{l:concentrate},
is the special case with $W=0$ and $\Im\lambda\geq ch$).
See~\cite[\S8.2]{nhp} for a proof using positive commutator estimates directly
instead of going through semiclassical measures.
One could also replace $\Re\sigma(W)$ in~\eqref{e:cim} by its
finite time avarage average along the flow of $H_p$ on $K\cap \{p=0\}$,
see for instance~\cite[Theorem~3.2]{royer}.
\begin{lemm}
  \label{l:noupper}
Fix small $\delta>0$; we use the set $U_\delta$ defined in~\eqref{e:U-delta}.
Assume that $A,B,B_1\in\Psi^{\comp}_h(X)$ satisfy (see Figure~\ref{f:lil})
\begin{itemize}
\item $\WFh(A)\subset U_{3\delta/2}$ and
$A=1+\mathcal O(h^\infty)$ microlocally on $\overline{U_{\delta}}$;
\item $\WFh(B)\subset U_{3\delta}\cap \{|\varphi_+|>\delta/2\}$ and
$B=1+\mathcal O(h^\infty)$ microlocally on\break $\overline{U_{2\delta}}\cap \{|\varphi_+|\geq\delta\}$;
\item $\WFh(B_1)\subset U_{3\delta}$
and $B_1=1+\mathcal O(h^\infty)$ microlocally on $\overline{U_{2\delta}}$.
\end{itemize}
Take $W\in\Psi^{\comp}_h(X)$ and $\lambda=\mathcal O(h)$ such that
\begin{equation}
  \label{e:cim}
\Re \sigma(W)+h^{-1}\Im\lambda\geq c>0\quad\text{on }U_{3\delta}
\end{equation}
for some constant $c$. Then for each $u\in L^2(X)$,
$$
\|Au\|_{L^2}\leq Ch^{-1}\|B_1(P-ihW-\lambda)u\|_{L^2}+C\|Bu\|_{L^2}+\mathcal O(h^\infty)\|u\|_{L^2}.
$$
\end{lemm}
\begin{figure}
\includegraphics{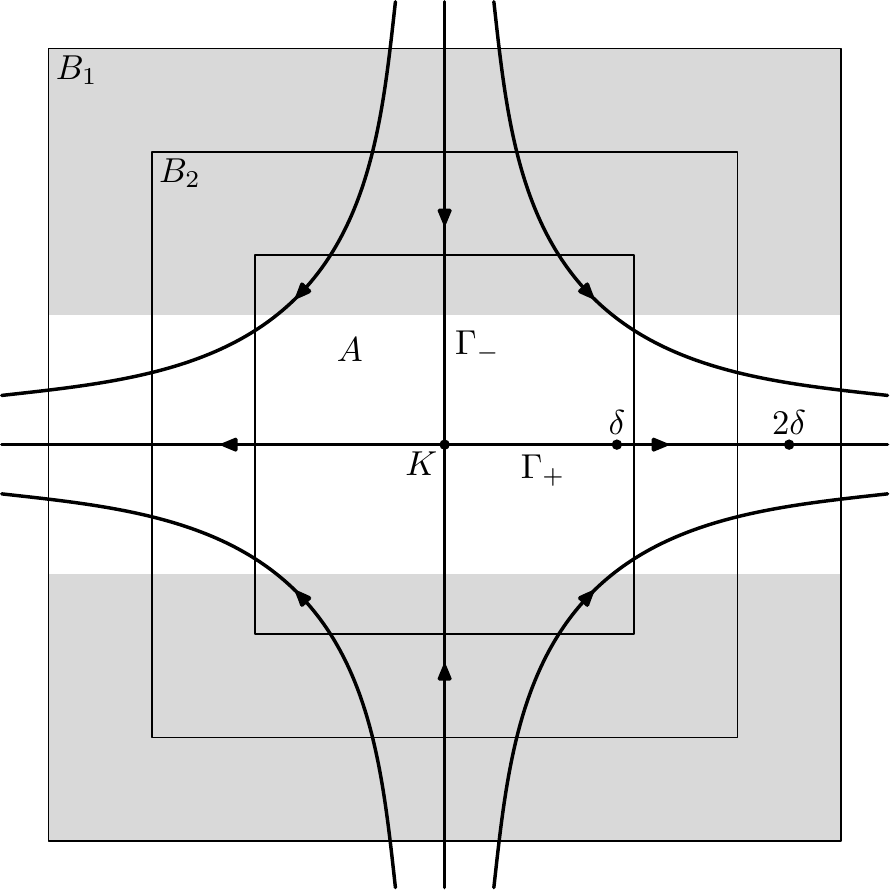}
\caption{The flow $H_p$ and the wavefront sets of $A,B,B_1,B_2$; here $\WFh(B)$ is the
shaded portion of $\WFh(B_1)$. The horizontal
axis is $\varphi_-$ and the vertical axis is $\varphi_+$.}
\label{f:lil}
\end{figure}
\begin{proof}
Take $B_2\in\Psi^{\comp}_h(X)$ such that
$\WFh(B_2)\subset U_{2\delta}$ and
$B_2=1+\mathcal O(h^\infty)$ microlocally on $\overline{U_{3\delta/2}}$. We first claim
that for each fixed $\varepsilon_0>0$, there exists a constant $C$
such that for $h$ small enough,
\begin{equation}
  \label{e:fuzzy1}
\|Au\|_{L^2}\leq Ch^{-1}\|B_2(P-ihW-\lambda)u\|_{L^2}+C\|Bu\|_{L^2}
+\varepsilon_0\|u\|_{L^2}.
\end{equation}
We argue by contradiction. If~\eqref{e:fuzzy1} does not hold, then
there exist sequences $h_j\to 0$, $\lambda=\lambda(h_j)$,
and $u=u(h_j)\in L^2$ such that $\|u\|_{L^2}\leq 1$,
$(P-ihW-\lambda)u=o(h)$ microlocally in $U_{3\delta/2}$, and
$\|Bu\|_{L^2}=o(1)$, yet $\|Au\|_{L^2}$ is separated
away from zero.

By passing to a subsequence, we may assume
that $u$ converges to some measure $\mu$ in the sense of Definition~\ref{d:measures}.
Since for each $(x,\xi)\in U_{3\delta/2}\setminus\Gamma_+$, there exists
$t\geq 0$ such that $e^{-tH_p}(x,\xi)\in \{\sigma(B)=1\}$
and $e^{-sH_p}(x,\xi)\in U_{3\delta/2}$ for $s\in [0,t]$, by propagation
of singularities~\cite[Proposition~2.5]{zeta} we see that
$u=o(1)$ microlocally in $U_{3\delta/2}\setminus \Gamma_+$ and thus 
\begin{equation}
  \label{e:supp}
\mu\big(U_{3\delta/2}\setminus\Gamma_+\big)=0.
\end{equation}
Passing to a subsequence, we may assume that $h^{-1}\lambda\to \omega\in\mathbb C$.
By Lemma~\ref{l:typical},
for each $a\in C_0^\infty(U_{3\delta/2})$,
$$
\int_{T^*X} H_pa\,d\mu=2\int_{T^*X} (\Re\sigma(W)+\Im\omega)a\,d\mu.
$$
By~\eqref{e:U-delta-map}, \eqref{e:supp}, and~\eqref{e:cim}, it follows
that for each nonnegative $a\in C_0^\infty(U_{3\delta/2})$,
$$
\int_{\Gamma_+\cap U_{3\delta/2}}a\circ e^{tH_p}\,d\mu\geq e^{2ct}\int_{\Gamma_+\cap U_{3\delta/2}} a\,d\mu,\quad t\geq 0.
$$
However, the left-hand side is bounded uniformly as $t\to +\infty$ (by $\|a\|_{\infty}\cdot\mu(U_{3\delta/2})$),
therefore $\mu(U_{3\delta/2})=0$ and thus $\|Au\|_{L^2}=o(1)$, giving a contradiction
and proving~\eqref{e:fuzzy1}.

Take $B'_2\in\Psi^{\comp}_h(X)$ such that
$B'_2=1+\mathcal O(h^\infty)$ microlocally on $\WFh(B_2)$.
We apply~\eqref{e:fuzzy1}
to $B'_2u$ to get for each $\varepsilon_0>0$,
\begin{equation}
  \label{e:fuzzy2}
\|Au\|_{L^2}\leq Ch^{-1}\|B_2(P-ihW-\lambda)u\|_{L^2}+C\|Bu\|_{L^2}+\varepsilon_0 \|B'_2u\|_{L^2}+\mathcal O(h^\infty)\|u\|_{L^2}.
\end{equation}
Now, for a correct choice of $B'_2$, for each $(x,\xi)\in \WFh(B'_2)\cap \{p=0\}$, there exists
$t\geq 0$ such that $e^{-tH_p}(x,\xi)\in \{\sigma(A)\neq 0\}\cup \{\sigma(B)\neq 0\}$
and $e^{-sH_p}(x,\xi)\in \{\sigma(B_1)\neq 0\}$ for $s\in [0,t]$. By propagation of singularities,
we then have
\begin{equation}
  \label{e:fuzzy3}
\|B'_2u\|_{L^2}\leq Ch^{-1}\|B_1(P-ihW-\lambda)u\|_{L^2}
+C\|Bu\|_{L^2}+C\|Au\|_{L^2}+\mathcal O(h^\infty)\|u\|_{L^2}.
\end{equation}
It remains to combine~\eqref{e:fuzzy2} with~\eqref{e:fuzzy3} and take $\varepsilon_0$ small enough.
\end{proof}

\section{First gap}
\label{s:gap1}

In this section, we prove Theorem~\ref{t:gap1}. We argue by contradiction.
Assume that the estimate~\eqref{e:main} does not hold; then (replacing $\varepsilon$
by $2\varepsilon$)
there exist sequences $h_j\to 0$, $\lambda=\lambda(h_j)$, and $u=u(h_j)\in L^2(X)$ such that
$$
\|u\|_{L^2}=1,\quad
\|(P-iQ-\lambda)u\|_{L^2}=\mathcal O(h^2),\quad
|\lambda|=\mathcal O(h),\quad
\Im\lambda>-(\nu_{\min}-2\varepsilon)h/2.
$$
By Lemma~\ref{l:concentrate}, we have
\begin{equation}
  \label{e:fg-4}
u=\mathcal O(h)\quad\text{microlocally on }T^*X\setminus\Gamma_+.
\end{equation}
Take small $\delta>0$ and let $U_\delta\subset T^*X$ be defined in~\eqref{e:U-delta}.
Recall that $Q=\mathcal O(h^\infty)$ microlocally on $U_{3\delta}$; therefore,
\begin{equation}
  \label{e:fg-1}
(P-\lambda)u=\mathcal O(h^2)\quad\text{microlocally on }U_{3\delta}.
\end{equation}
Take $\Theta_+,W_+\in\Psi^{\comp}_h(X)$ such that
$$
\sigma(\Theta_+)=\varphi_+,\quad
\sigma(W_+)=c_+,\quad
\Theta_+^*=\Theta_+,
$$
where $\varphi_+,c_+$ are constructed in Lemma~\ref{l:phi+}.
\begin{lemm}
\label{l:theta-killed}
For
$\Im\lambda>-(\nu_{\min}-2\varepsilon)h$ (which is weaker
than the assumption of Theorem~\ref{t:gap1}),
\begin{equation}
  \label{e:fg-5}
\Theta_+u=\mathcal O(h)\quad\text{microlocally on }U_\delta.
\end{equation}
\end{lemm}
\begin{proof}
Since
$H_p\varphi_+=-c_+\varphi_+$ on $U_{3\delta}$, we have
\begin{equation}
  \label{e:fg-2}
[P,\Theta_+]=ihW_+\Theta_++\mathcal O(h^2)_{\Psi^{\comp}_h}\quad\text{microlocally on }
U_{3\delta}.
\end{equation}
Applying $\Theta_+$ to~\eqref{e:fg-1} and using~\eqref{e:fg-2}, we obtain
\begin{equation}
  \label{e:fg-3}
(P-ihW_+-\lambda)\Theta_+u=\mathcal O(h^2)\quad\text{microlocally on }U_{3\delta}.
\end{equation}
We have $\sigma(W_+)=c_+\geq\nu_{\min}-\varepsilon$ on $U_{3\delta}$;
therefore, $\sigma(W_+)+h^{-1}\Im\lambda\geq \varepsilon>0$ on $U_{3\delta}$.
It remains to apply Lemma~\ref{l:noupper} to $\Theta_+u$
and use \eqref{e:fg-4} and~\eqref{e:fg-3}.
\end{proof}
Now, passing to a subsequence of $h_j$, we may assume that (in the sense of Definition~\ref{d:measures})
$$
u\to \mu,\quad
h^{-1}\Im\lambda\to \omega\in\mathbb C.
$$
By Lemma~\ref{l:concentrate}, we have
\begin{equation}
  \label{e:yippie}
\mu(U_\delta)>0,\quad
\mu(U_\delta\setminus\Gamma_+)=0.
\end{equation}
We now show that $\mu$ is Lipschitz in the direction transversal to $K$.
Recall that $\varphi_-$, constructed in Lemma~\ref{l:phi+}, is
a defining function of $K$ on $\Gamma_+\cap U_\delta$.

\begin{figure}
\includegraphics{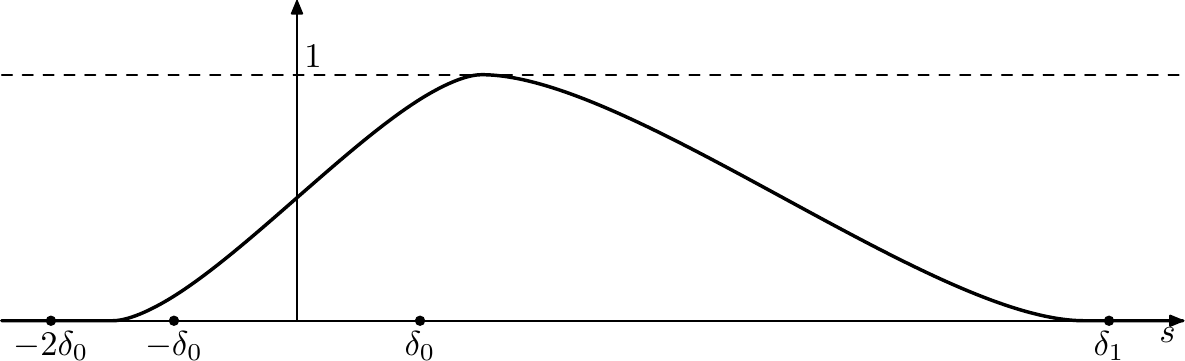}
\caption{The function $\tilde a$ used in the proof of Lemma~\ref{l:lipschiz}. Here
$\delta_1>0$ is fixed and $\delta_0\to 0$.}
\label{f:a}
\end{figure}

\begin{lemm}
\label{l:lipschiz}
There exists a constant $C$ such that for each $\delta_0>0$,
\begin{equation}
  \label{e:lipschiz}
\mu\big(U_\delta\cap \{|\varphi_-|<\delta_0\}\big)\leq C\delta_0.
\end{equation}
\end{lemm}
\begin{proof}
We may assume that $\delta_0$ is arbitrarily small.
Applying Lemma~\ref{l:typical} to~\eqref{e:fg-5}, we see that
there exists a constant $C$ such that for each $a\in C_0^\infty(\Gamma_+\cap U_{\delta})$,
\begin{equation}
  \label{e:uni}
\bigg|\int_{\Gamma_+\cap U_\delta} H_{\varphi_+} a\,d\mu\bigg|\leq C\|a\|_{\infty}.
\end{equation}
To see~\eqref{e:lipschiz}, we need to apply this estimate to
a function $a$ depending on $\delta_0$ which, written in coordinates $(\rho,s)$ of~\eqref{e:funny-coordinates},
has the form $a=\tilde a(s)\chi(\rho)$ with $\chi\in C_0^\infty(K\cap \{|p|<\delta\})$
equal to 1 on $K\cap \{p=0\}$ and
$$
\begin{gathered}
\supp \tilde a\subset (-2\delta_0, \delta_1),\quad
\|\tilde a\|_{\infty}\leq 1,\quad
\partial_s \tilde a\geq -{2\over\delta_1},
\\
\partial_s \tilde a\geq {1\over 3\delta_0}\quad\text{for }|s|\leq \delta_0
\end{gathered}
$$
for some fixed small $\delta_1\in (0,\delta)$ independent of $\delta_0$, see Figure~\ref{f:a}.
Then
$$
\int_{\Gamma_+\cap U_\delta} H_{\varphi_+}a\,d\mu
\geq {1\over 3\delta_0} \mu\big(\{|s|\leq\delta_0\}\big)
-{2\over \delta_1}\mu\big(\{\delta_0\leq |s|\leq \delta_1\}\big)
$$
The left-hand side is bounded by a $\delta_0$-independent constant by~\eqref{e:uni};
so is the second term on the right-hand side (since $\mu$ is a finite measure).
Multiplying both sides by $3\delta_0$
and using that $\Gamma_+\cap U_\delta\cap \{|\varphi_-|\leq \delta_0\}
\subset \{|s|\leq C\delta_0\}$ for some constant $C$,
we obtain~\eqref{e:lipschiz}.
\end{proof}

We now have by Lemma~\ref{l:typical} applied to~\eqref{e:fg-1},
for each $a\in C_0^\infty(U_\delta)$,
$$
\int_{\Gamma_+\cap U_\delta} H_pa\,d\mu=2\Im\omega \int_{\Gamma_+\cap U_\delta}a\,d\mu.
$$
For $t\geq 0$, since $e^{-tH_p}(\Gamma_+\cap U_\delta)\subset\Gamma_+\cap U_\delta$,
we have
\begin{equation}
  \label{e:measure-below}
\mu(e^{-tH_p}(U_\delta))=e^{2t\Im\omega}\mu(U_\delta).
\end{equation}
However, by~\eqref{e:U-delta-map}, \eqref{e:yippie}, and~\eqref{e:lipschiz}, as $t\to +\infty$
$$
\mu(e^{-tH_p}(U_\delta))\leq \mu\big(U_\delta\cap \{|\varphi_-|<\delta e^{-(\nu_{\min}-\varepsilon)t}\}\big)\leq Ce^{-(\nu_{\min}-\varepsilon)t}.
$$
However, since $\Im\omega\geq -(\nu_{\min}-2\varepsilon)/2$,
we see that $e^{2t\Im\omega}$ decays exponentially slower
than $e^{-(\nu_{\min}-\varepsilon)t}$ as $t\to +\infty$.
Since $\mu(U_\delta)>0$
by~\eqref{e:yippie}, we arrive to a contradiction, finishing the proof of Theorem~\ref{t:gap1}.


\section{Further gaps for $r$-normally hyperbolic trapping}
\label{s:gap2}

In this section, we prove Theorem~\ref{t:gap2}.
For small $\delta>0$, let $U_{\delta}$ be defined by~\eqref{e:U-delta}.

\subsection{Auxiliary microlocal construction}
The proof relies on the following statement, which in particular makes it possible
to improve the remainder in~\eqref{e:fg-2} to $\mathcal O(h^\infty)$.
For $Z=0$, the operators $\Theta,W$ below were previously
constructed in~\cite[Proposition~7.1]{nhp}.
\begin{lemm}
  \label{l:theta-more}
Assume that the trapping is $r$-normally hyperbolic for all $r$, in the sense of~\eqref{e:r-nh}.
Then for each $Z\in\Psi^{\comp}_h(X)$, there exist
$\Theta,W\in \Psi^{\comp}_h(X)$
such that
\begin{gather}
  \label{e:theta-more}
\Theta(P-ihZ)=(P-ih W)\Theta+\mathcal O(h^\infty)\quad\text{microlocally on }U_{3\delta};\\
  \label{e:theta-more-2}
\sigma(\Theta)=\varphi_+,\quad
\sigma(W)=\sigma(Z)+c_+\quad\text{on }U_{3\delta}.
\end{gather}
\end{lemm}
\begin{proof}
We construct operators
$$
\Theta_k,W_k\in\Psi^{\comp}_h(X),\quad
k=0,1,\dots,
$$
such that microlocally in $U_{3\delta}$, we have for some $R_{k}\in \Psi^{\comp}_h(X)$,
\begin{equation}
  \label{e:commu}
\Theta_k(P-ihZ)=(P-ihW_k)\Theta_{k}+h^{k+2}R_k+\mathcal O(h^\infty)
\end{equation}
and take $\Theta,W$ to be the asymptotic limits as $k\to +\infty$ of these sequences
of operators.

We use induction on $k$. For $k=0$, take any $\Theta_0,W_0$ satisfying in $U_{3\delta}$,
$$
\sigma(\Theta_0)=\varphi_+,\quad
\sigma(W_0)=\sigma(Z)+c_+,
$$
then~\eqref{e:commu} follows immediately from~\eqref{e:identities} and the fact that
$H_p\varphi_+=-c_+\varphi_+$ on $U_{3\delta}$.

Assume now that $k>0$ and~\eqref{e:commu} holds for $k-1$; we construct the operators
$\Theta_k$ and $W_k$ in the form
$$
\Theta_k:=\Theta_{k-1}+h^k\Omega_k,\quad
W_k:=W_{k-1}+h^k Y_k,
$$
where $\Omega_k,Y_k\in\Psi^{\comp}_h(X)$. We rewrite~\eqref{e:commu}
for $k$ as the following statement, valid microlocally in $U_{3\delta}$:
$$
[P,\Omega_k]+ih(\Omega_kZ- W_{k-1}\Omega_k
-Y_k\Theta_{k-1})=hR_{k-1}+\mathcal O(h^2)_{\Psi^{\comp}_h}.
$$
and this translates to the following equation on the principal symbols%
\footnote[2]{This is the place where we cannot take $P$ to act on sections of
a vector bundle, as $R_{k-1}$ and thus $\Omega_k,Y_k$ need not be principally
scalar and a more complicated transport equation would be required.}
of $\Omega_k,Y_k$
in $U_{3\delta}$:
$$
(H_p+c_+)\sigma(\Omega_k)+\varphi_+\sigma(Y_k)=i\sigma(R_{k-1}).
$$
Since we can choose $\sigma(Y_k)$ arbitrarily and $\varphi_+$ is a defining function of $\Gamma_+$, it suffices to
find $\Omega_k$ whose symbol solves the transport equation
$$
(H_p+c_+)\sigma(\Omega_k)=i\sigma(R_{k-1})\quad\text{on }\Gamma_+\cap U_{3\delta}.
$$
The existence of such a symbol follows from Lemma~\ref{l:transport}
(which is where we use $r$-normal hyperbolicity), finishing the construction
of $\Theta_k,W_k$.
\end{proof}

Iterating Lemma~\ref{l:theta-more}, we obtain
\begin{lemm}
Assume that the trapping
is $r$-normally hyperbolic for all $r$.
Then there exist operators
$$
\Theta^+_m,W^+_m\in \Psi^{\comp}_h(X),\quad m=0,1,\dots
$$
such that $W^+_0=0$ and
\begin{equation}
  \label{e:theta-more-3}
\Theta^+_m(P-ihW^+_m)=(P-ih W^+_{m+1})\Theta^+_m+\mathcal O(h^\infty)\quad\text{microlocally on }U_{3\delta}.
\end{equation}
Moreover, we have
$$
\sigma(\Theta^+_m)=\varphi_+,\quad
\sigma(W^+_m)=mc_+\quad\text{on }U_{3\delta}.
$$
\end{lemm}

\subsection{Measures associated to resonant states and proof of Theorem~\ref{t:gap2}}

Assume that $u=u(h_j)\in L^2(X)$ is bounded uniformly in $h$ and define
$u_0,u_1,\ldots\in L^2(X)$ by the relations
\begin{equation}
  \label{e:u-m}
u_0:=u,\quad
u_{m+1}:=\Theta^+_m u_m.
\end{equation}
\begin{lemm}
  \label{l:refinery}
Assume that for some $\alpha\geq 0$ and $\lambda=\mathcal O(h)$,
\begin{equation}
  \label{e:condic0}
\|(P-iQ-\lambda)u\|_{L^2}=o(h^{\alpha+1}).
\end{equation}
Then for each $m\geq 0$,
\begin{gather}
  \label{e:condic1}
(P-ihW^+_m-\lambda)u_m=o(h^{\alpha+1})\quad\text{microlocally on }U_{3\delta};\\
  \label{e:condic2}
u_m=o(h^\alpha)\quad\text{microlocally on }U_{3\delta}\setminus (\Gamma_+\cap \{p=0\}).
\end{gather}
Moreover, if
\begin{equation}
  \label{e:lambda-con}
\Im\lambda>-m(\nu_{\min}-\varepsilon)h+\varepsilon h,
\end{equation}
then $u_m=o(h^{\alpha})$ microlocally on $U_\delta$.
\end{lemm}
\begin{proof}
The condition~\eqref{e:condic1} for $m=0$ follows directly from~\eqref{e:condic0}
(since $Q=\mathcal O(h^\infty)$ microlocally on $U_{3\delta}$);
for $m>0$, it follows by induction from~\eqref{e:theta-more-3}. The condition~\eqref{e:condic2}
for $m=0$ follows from Lemma~\ref{l:concentrate}, and is still true for $m\geq 0$
since $u_m$ is obtained by applying an $h$-pseudodifferential operator to $u$.

The last statement of this lemma follows from~\eqref{e:condic1}, \eqref{e:condic2}
by Lemma~\ref{l:noupper}, as on $U_{3\delta}$, we have
$\sigma(W^+_m)=mc_+\geq m(\nu_{\min}-\varepsilon)$ and thus $\sigma(W^+_m)+h^{-1}\Im\lambda\geq \varepsilon >0$.
\end{proof}
We now obtain information on semiclassical measures of solutions to $\Theta^+_m v=o(h)$,
which generalizes the Lipschitz property of Lemma~\ref{l:lipschiz}. It is possible
to give a more precise description of the measure along the Hamiltonian flow
lines of $\varphi_+$; for $\Theta^+_0$, this is done in~\cite[\S8.5, Theorem~4]{nhp}.
\begin{lemm}
  \label{l:lab}
Assume that $v=v(h_j)\in L^2(X)$
converges to some measure $\mu$ in the sense of Definition~\ref{d:measures},
and $\mu(U_\delta)>0$, $\mu(U_\delta\setminus\Gamma_+)=0$.
Assume also that for some $m$,
$$
\Theta^+_mv=o(h)\quad\text{microlocally in }U_\delta.
$$
Then for some constant $C$ and for each $\delta_0\in(0,1)$, we have
$$
C^{-1}\delta_0\leq \mu\big(U_\delta\cap \{|\varphi_-|<\delta_0\}\big)\leq C\delta_0.
$$
\end{lemm}
\begin{proof}
Let $b:=\sigma(h^{-1}\Im \Theta^+_m)$, where we denote
$\Im A:={1\over 2i}(A-A^*)$. 
By Lemma~\ref{l:typical}, we see that for each $a\in C_0^\infty(U_\delta)$,
$$
\int_{U_\delta\cap \Gamma_+} (H_{\varphi_+}+2b) a\,d\mu=0.
$$
It remains to note that in the coordinates~\eqref{e:funny-coordinates},
we have for each $a\in C_0^\infty(U_\delta)$
$$
\int a\,d\mu=\int e^{2\psi}a \,d\tilde\mu(\rho) ds
$$
where $\psi\in C^\infty(U_\delta\cap \Gamma_+)$ is some fixed solution to the equation
$H_{\varphi_+}\psi=b$ and $\tilde\mu$ is some measure on $K$.
\end{proof}
Lemma~\ref{l:lab} makes it possible to show
bounds on approximate solutions to $\Theta^+_mv=0$ and $(P-ihW_m^+-\lambda)v=0$
for $\Im\lambda$ outside of a certain interval:
\begin{lemm}
  \label{l:lab2}
Let $A,B_1\in\Psi^{\comp}_h(X)$ satisfy
\begin{itemize}
\item $\WFh(A)\subset U_{3\delta/2}$ and $A=1+\mathcal O(h^\infty)$ microlocally on $\overline{U_{\delta}}$;
\item $\WFh(B_1)\subset U_{3\delta}$ and $B_1=1+\mathcal O(h^\infty)$ microlocally on $\overline{U_{2\delta}}$.
\end{itemize}
Fix $m$ and assume that $\lambda=\mathcal O(h)$ is such that
\begin{equation}
  \label{e:luo}
h^{-1}\Im\lambda\notin [-(m+1/2)(\nu_{\max}+\varepsilon)-\varepsilon,-(m+1/2)(\nu_{\min}-\varepsilon)+\varepsilon].
\end{equation}
Then for each $v\in L^2(X)$,
$$
\|Av\|_{L^2}\leq Ch^{-1}\big(\|B_1(P-ihW_m^+-\lambda)v\|_{L^2}+\|B_1 \Theta^+_mv\|_{L^2}\big)
+\mathcal O(h^\infty)\|v\|_{L^2}.
$$
\end{lemm}
\begin{proof}
Arguing similarly to the proof of Lemma~\ref{l:noupper} (removing $B$ from~\eqref{e:fuzzy1} and~\eqref{e:fuzzy2}
and replacing $B$ by $B_1\Theta_m^+$ in~\eqref{e:fuzzy3}),
we see that it suffices to show
that if $h_j\to 0$ and $v=v(h_j)$ satisfies $\|v\|_{L^2}\leq 1$ and
\begin{gather}
  \label{e:roo-1}
(P-ihW_m^+-\lambda)v=o(h)\quad\text{microlocally in }U_{3\delta/2},\\
  \label{e:roo-2}
\Theta^+_mv=o(h)\quad\text{microlocally in }U_{3\delta/2},
\end{gather}
then $v=o(1)$ microlocally in $U_{3\delta/2}$.

Passing to a subsequence, we may assume that $v$ converges to some measure $\mu$ in
the sense of Definition~\ref{d:measures} and $h^{-1}\lambda\to \omega\in\mathbb C$, where
\begin{equation}
  \label{e:luo2}
\Im\omega\notin [-(m+1/2)(\nu_{\max}+\varepsilon)-\varepsilon,-(m+1/2)(\nu_{\min}-\varepsilon)+\varepsilon].
\end{equation}
We have $\mu(U_{3\delta/2}\setminus \Gamma_+)=0$
by Lemma~\ref{l:measure-elliptic} applied to~\eqref{e:roo-2}.
By Lemma~\ref{l:typical} applied to~\eqref{e:roo-1}, we see that for each $a\in C_0^\infty(U_{3\delta/2})$,
$$
\int_{\Gamma_+\cap U_{3\delta/2}} H_pa\,d\mu=2\int_{\Gamma_+\cap U_{3\delta/2}} (\Im\omega+\sigma(W_m^+))a\,d\mu;
$$
for $t\geq 0$, since $e^{-tH_p}(\Gamma_+\cap U_{3\delta/2})\subset\Gamma_+\cap U_{3\delta/2}$,
we have by~\eqref{e:luo2} either (when $\Im\omega$ is too big)
\begin{equation}
  \label{e:sar1}
\mu(e^{-tH_p}(\Gamma_+\cap U_{3\delta/2}))\geq e^{-(\nu_{\min}-3\varepsilon)t}\mu(U_{3\delta/2}) 
\end{equation}
or (when $\Im\omega$ is too small)
\begin{equation}
  \label{e:sar2}
\mu(e^{-tH_p}(\Gamma_+\cap U_{3\delta/2}))\leq e^{-(\nu_{\max}+3\varepsilon)t}\mu(U_{3\delta/2}).
\end{equation}
Here we used that $\sigma(W_m^+)=mc_+\in [m(\nu_{\min}-\varepsilon),m(\nu_{\max}+\varepsilon)]$ on $U_{3\delta/2}$.
Using Lemma~\ref{l:lab} with~\eqref{e:roo-2} and observing that for $t\geq 0$, by Lemma~\ref{l:phi+}
$$
\begin{gathered}
\{|\varphi_-|< {\textstyle{3\delta\over 2}}e^{-(\nu_{\max}+\varepsilon)t}\}\cap \Gamma_+\cap U_{3\delta/2}
\quad\subset\quad e^{-tH_p}(\Gamma_+\cap U_{3\delta/2}),\\
e^{-tH_p}(\Gamma_+\cap U_{3\delta/2})\quad\subset\quad \{|\varphi_-|< {\textstyle {3\delta\over 2}}e^{-(\nu_{\min}-\varepsilon)t}\}\cap \Gamma_+\cap U_{3\delta/2}
\end{gathered}
$$
both~\eqref{e:sar1} and~\eqref{e:sar2} imply, by taking the limit $t\to +\infty$, that $\mu(U_{3\delta/2})=0$ and
thus $v=o(1)$ microlocally on $U_{3\delta/2}$, finishing the proof.
\end{proof}
We can now prove Theorem~\ref{t:gap2}. We argue by contradiction. If~\eqref{e:gap2}
does not hold, then (changing $\varepsilon$) there exist sequences $h_j\to 0$
and $u=u(h_j)\in L^2(X)$ such that
$$
\|u\|_{L^2}=1,\quad
\|(P-iQ-\lambda)u\|_{L^2}=o(h^{m+2}),
$$
where $\lambda=\lambda(h_j)=\mathcal O(h_j)$ satisfies
\begin{equation}
  \label{e:im-cond}
h^{-1}\Im\lambda \in \big(-(m+1/2)(\nu_{\min}-\varepsilon)+\varepsilon,-(m-1/2)(\nu_{\max}+\varepsilon)-\varepsilon\big).
\end{equation}
Define $u_0,\dots, u_{m+1}$
by~\eqref{e:u-m}; then by Lemma~\ref{l:refinery},
$u_{m+1}=o(h^{m+1})$ microlocally in $U_{\delta}$. Here we used the following corollary of~\eqref{e:im-cond}:
$$
h^{-1}\Im\lambda>-(m+1)(\nu_{\min}-\varepsilon)+\varepsilon.
$$
Using Lemma~\ref{l:lab2} for
$v=u_m,u_{m-1},\dots, u_0$ (and a decreasing sequence of $\delta$'s),
we get $u_j=o(h^{j})$ microlocally in $U_{3^{j-m-1}\delta}$
for $j=0,1,\dots, m+1$. Here we used the following corollary of~\eqref{e:im-cond}:
$$
h^{-1}\Im\lambda\notin [-(j+1/2)(\nu_{\max}+\varepsilon)-\varepsilon,
-(j+1/2)(\nu_{\min}-\varepsilon)+\varepsilon],\quad 0\leq j\leq m.
$$
In particular, we see that $u=u_0=o(1)$ microlocally in the neighborhood
$U_{3^{-m-1}\delta}$ of $K\cap \{p=0\}$, giving a contradiction
with Lemma~\ref{l:concentrate} and finishing the proof of Theorem~\ref{t:gap2}.

\section{A one-dimensional example}
\label{s:lower}

In this section, we show a lower bound on the cutoff resolvent for $\lambda=0$ and a specific operator
with normally hyperbolic trapping. We start with the model operator (which we later
put into the framework discussed in the introduction)
$$
P_0=xhD_x+{h\over 2i}:C^\infty(\mathbb R)\to C^\infty(\mathbb R),\quad
D:={1\over i}\partial.
$$
Note that $P_0^*=P_0$.

We construct an approximate solution $u_0$ for the
equation $P_0u=0$ by truncating the exact solutions $(x_\pm)^{-1/2}$ in the frequency space.
More precisely, take a function $\chi\in \mathscr S(\mathbb R)$ such that $\hat\chi\in C_0^\infty\big((1/2,1)\big)$ and
define
$$
\psi(x)=\sgn x\int_0^x |x|^{-1/2}|y|^{-1/2}\chi(y)\,dy.
$$
Here $\hat\chi$ denotes the Fourier transform of $\chi$:
$$
\hat\chi(\xi)=\int_{\mathbb R} e^{-ix\xi}\chi(x)\,dx.
$$
Then (as can be seen by considering the Taylor expansion of $\chi$ at 0)
we have
$$
\psi\in C^\infty(\mathbb R),\quad
P_0\psi=-ih\chi.
$$
Moreover, we have
\begin{equation}
  \label{e:asympsi}
\psi(x)=\psi_\pm |x|^{-1/2}+\mathcal O(|x|^{-\infty})\quad\text{as }x\to \pm\infty,
\end{equation}
where the constants $\psi_\pm$ are given by
$$
\psi_\pm=\pm \int_0^{\pm\infty} |y|^{-1/2}\chi(y)\,dy.
$$
We choose $\chi$ so that $\psi_\pm\neq 0$; this is possible since
$$
\psi_\pm={e^{\pm {i\pi\over 4}}\over 2\sqrt{\pi}}\int_0^\infty \xi^{-1/2}\hat\chi(\xi)\,d\xi.
$$
Define
$$
u_0(x;h):=h^{-1/2}\psi(x/h),\quad
f_0(x;h):=-ih^{1/2}\chi(x/h),
$$
then $P_0u_0=f_0$. We have
\begin{equation}
  \label{e:wf-u0}
\WFh(u_0)\subset \{\xi=0\}\cup \{x=0,\ \xi\in [0,1)\}.
\end{equation}
Indeed, if $a\in C_0^\infty\big((0,\infty)\big)$, then by~\eqref{e:asympsi},
$$
a(x)u_0(x)=\psi_+ a(x)x^{-1/2}+\mathcal O(h^\infty);
$$
since $\psi_+ a(x)x^{-1/2}$ is smooth and independent of $h$, we see by~\cite[\S8.4.2]{e-z}
that $\WFh(au_0)\subset \{\xi=0\}$. Same is true when $a\in C_0^\infty\big((-\infty,0)\big)$. Moreover, since $(x\partial_x+{1\over 2})\psi=\chi$,
we have $(\xi\partial_\xi+{1\over 2})\hat\psi=-\hat\chi$; since $\partial_x\psi\in L^2$, we have
$\xi\hat\psi(\xi)\in L^2$ and the fact that
$\supp\hat\chi\subset (1/2,1)$ implies $\supp\hat\psi\subset [0,1)$.
This in turn implies that $\WFh(u_0)\subset \{\xi\in [0,1)\}$, finishing the proof of~\eqref{e:wf-u0}.
Since $\supp\hat\chi\subset (1/2,1)$ and $\chi$ is Schwartz, we also have
\begin{equation}
  \label{e:wf-f0}
\WFh(f_0)\subset \{x=0,\ \xi\in (1/2,1)\}.
\end{equation}
We furthermore calculate using~\eqref{e:asympsi}
\begin{equation}
\label{e:norm-bounds}
\|u_0\|_{L^2(-1,1)}=\sqrt{|\psi_+|^2+|\psi_-|^2}\sqrt{\log(1/h)}+\mathcal O(1),\quad
\|f_0\|_{L^2(\mathbb R)}=\mathcal O(h). 
\end{equation}

We now put $u_0,f_0,P_0$ into the framework discussed in the introduction. Put
$$
X=\mathbb S^1=\mathbb R/(6\mathbb Z);
$$
we view $X$ as the interval $[-3,3]$ with the endpoints glued together.
We consider an operator
$$
P\in\Psi^{\comp}_h(X),\quad
P=P^*;\quad
P=P_0\quad\text{microlocally on }\{|x|\leq 2,\ |\xi|\leq 2\}.
$$
Note that by~\eqref{e:wf-u0}, we have $Pu_0=f_0+\mathcal O(h^\infty)_{C^\infty}$
on $(-2,2)$.

We next consider $q_0\in C^\infty(X)$ such that $q_0=0$ on $[-1,1]$,
$q_0\geq 0$ everywhere, and $q_0=1$ outside of $[-3/2,3/2]$.
Take $\chi_1\in C_0^\infty(-2,2)$ such that $\chi_1=1$ on $[-3/2,3/2]$
and define
$$
u_1=\chi_1 u_0,\quad
f_1=(P-iq_0)u_1-\chi_1 f_0=[P,\chi_1]u_0-iq_0\chi_1u_0+\mathcal O(h^\infty)_{C^\infty}.
$$
In particular, $f_1=\mathcal O(h^\infty)$ on $[-1,1]$ and $\WFh(f_1)\subset \{\xi=0\}$. Using ODE theory
for the operator $P_0-iq_0={h\over i}(x\partial_x+{1\over 2}+h^{-1}q_0)$, we construct
$u_2\in C^\infty(X)$ such that $u_2=0$ on $[-1,1]$,
$u_2=\mathcal O(e^{-\varepsilon/h})$ outside of $[-7/4, 7/4]$,
$\WFh(u_2)\subset \{\xi=0\}$, and
$$
(P-iq_0)u_2=-f_1+\mathcal O(h^\infty)_{C^\infty}.
$$
Finally, take a self-adjoint operator $Q_1\in\Psi^0_h$ such that $\WFh(Q_1)\subset \{|\xi|\geq 3/2\}$,
$\sigma(Q_1)\geq 0$ everywhere, and $\sigma(Q_1)>0$ for $|\xi|\geq 2$, denote
by $Q_0$ the multiplication operator by $q_0$, and put
$$
Q:=Q_0+Q_1,\quad
\sigma(Q)=q_0+\sigma(Q_1).
$$
Then the operator $P-iQ:L^2(X)\to L^2(X)$ satisfies the assumptions~\eqref{a:1}--\eqref{a:5}
of the introduction, with $\nu_{\min}=\nu_{\max}=1$ and
$$
\Gamma_+=\{\xi=0,\ |x|<2\},\quad
\Gamma_-=\{x=0,\ |\xi|<2\},\quad
K=\{x=\xi=0\}.
$$
We have
$$
(P-iQ)u=\chi_1 f_0+\mathcal O(h^\infty),\quad
u:=u_1+u_2.
$$
On the other hand, \eqref{e:norm-bounds} implies
$$
\|u\|_{L^2}\geq C^{-1}\sqrt{\log(1/h)},\quad
\|\chi_1f_0\|_{L^2}=\mathcal O(h).
$$
This means that for each $A\in\Psi^{\comp}_h(X)$ such that
$$
A=1+\mathcal O(h^\infty)\quad\text{microlocally on }V:=\{x=0,\ \xi\in (1/2,1)\},
$$
we have
$$
\|R(0)A\|_{L^2\to L^2}\geq C^{-1}h^{-1}\sqrt{\log(1/h)}.
$$
Note that $V\subset\Gamma_-$ and $V\cap\Gamma_+=\emptyset$, therefore we can choose
$A$ to be microlocalized away from the trapped set $K$.

\medskip\noindent\textbf{Acknowledgements.} I would like to thank Andr\'as Vasy for many conversations on
resolvent bounds which motivated this paper and provided several important insights,
and Maciej Zworski for invaluable advice throughout this project.
I am also grateful to Richard Melrose, St\'ephane Nonnenmacher,
and Kiril Datchev for very useful discussions, and to an anonymous referee
for suggestions to improve the manuscript.
This research was conducted during the period the author served as
a Clay Research Fellow; part of the work was completed while the author was
in residence at MSRI (NSF grant 0932078 000,
Fall 2013).

\def\arXiv#1{\href{http://arxiv.org/abs/#1}{arXiv:#1}}

\end{document}